\documentclass{amsart}
\usepackage{amsthm,amsmath,amssymb}
\usepackage[pdftex]{graphicx}
\usepackage{braket}
\usepackage{color}

\title[Knots that are not slice either in positons or in negatons]{Knots that are not slice either in positons or in negatons}
\author{Kouki Sato}
\date{}

\newtheorem*{question}{Question}
\newtheorem{thm}{Theorem}
\newtheorem{prop}{Proposition}
\newtheorem{lem}{Lemma}
\newtheorem{cor}{Corollary}

\theoremstyle{definition}

\newtheorem*{remark}{Remark}
\newtheorem*{acknowledge}{Acknowledgements}

\begin{document}
\maketitle

\begin{abstract}
An oriented compact 4-manifold $V$ with boundary $S^3$ 
is called a positon (resp.\ negaton)
if its intersection form is positive definite (resp.\ negative definite) 
and it is simply connected.
In this paper, we prove that there exist infinitely many knots 
which cannot bound null-homologous disks
either in positons or in negatons.
As a consequence, we find knots that cannot be unknotted
either by only positive crossing changes or by only negative crossing changes. 
\end{abstract}

\section{Introduction}
Throughout this paper we work in the smooth category,
and assume that all manifolds we deal with are compact, orientable and oriented.

Cochran, Harvey and Horn \cite{cochran-harvey-horn} introduced a filtration on the smooth knot concordance group, called the {\it bipolar filtration}.
They say that  
{\it a knot $K$ in $S^3$ is slice in $V$} for a 4-manifold $V$ with boundary $S^3$
if $K$ bounds a properly embedded disk $D$ in $V$ such that 
$[D,\partial D]= 0 \in H_2(V,\partial V;\mathbb{Z})$.
They observe slice knots in
simply-connected 4-manifolds with boundary $S^3$
whose intersection forms are positive definite (resp.\ negative definite), called {\it positons}
(resp.\ {\it negatons}).
Note that by the diagonalizability theorem of Donaldson and
the classification theorem of Freedman, any positon (resp.\ negaton) is 
homeomorphic to $n\mathbb{C}P^2 \setminus \mathring{B}^4$
(resp. $n\overline{\mathbb{C}P^2} \setminus \mathring{B}^4$) for some $n \in \mathbb{N}$.
Although the following question does not directly concern the bipolar filtration, it is natural to ask. 
\begin{question}
Does there exist a knot which is not slice either in positons or in negatons?
\end{question}
In this paper, we give an affirmative answer to this question.
In fact, we prove the following theorem.
\begin{thm}
\label{thm1}
For any knot $K$ with $\tau (K) <0$ and any integer $n$ with
$0 < n < -2\tau(K) + (\varepsilon(K) - 1)/2$, 
the cable knot $K_{2,2n+1}$ is not slice either in positons or in negatons.
\end{thm}
Here $K_{2,2n+1}$ denotes the $(2,2n+1)$-cable of $K$, $\tau(K) \in \mathbb{Z}$ is the 
Ozsv\'{a}th-Szab\'{o} $\tau$-invariant of $K$, and 
$\varepsilon(K) \in \{ 0, \pm 1 \}$ is Hom's $\varepsilon$-invariant of $K$.
Note that if $\tau(K) \neq 0$ then $\varepsilon(K) \neq 0$ (see \cite{hom}).

Next we discuss the relationship between Theorem \ref{thm1} and crossing changes.
If a knot  $K_1$ is deformed into $K_2$ by a crossing change 
from a positive crossing (Figure \ref{pos}) to a negative crossing (Figure \ref{neg})
(resp.\ from a negative crossing to a positive crossing),
then we say that $K_1$ is deformed into $K_2$
by a {\it positive (resp.\ negative) crossing change}. 
It is well known that there exist infinitely many knots 
which cannot be deformed into any slice knot by only positive crossing changes.
For instance, any knot with $\tau(K)<0$ has this property \cite{ozsvath-szabo2}.
In this paper, we obtain the following result as a corollary of Theorem \ref{thm1}.
\begin{cor}
\label{cor1}
For any knot $K$ with $\tau (K) <0$ and any integer $n$ with
$0 < n < -2\tau(K) + (\varepsilon(K) - 1)/2$, 
the cable knot $K_{2,2n+1}$ cannot be deformed into a slice knot
(in particular, it cannot be unknotted) either by only positive crossing
changes  or by only negative crossing changes.
\end{cor}

\begin{figure}[tbp]
\begin{minipage}[]{0.4\hsize}
\hspace{-1mm}
\includegraphics[scale = 0.7]{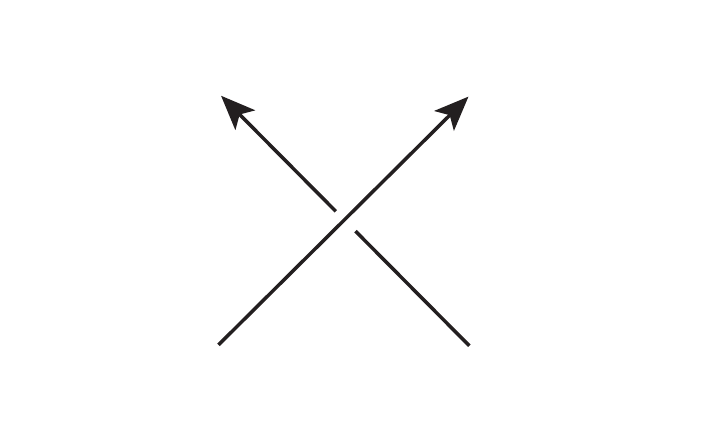}
\caption{\label{pos}}
 \end{minipage}
 \begin{minipage}[]{0.4\hsize}
\hspace{-4mm}
\includegraphics[scale = 0.7]{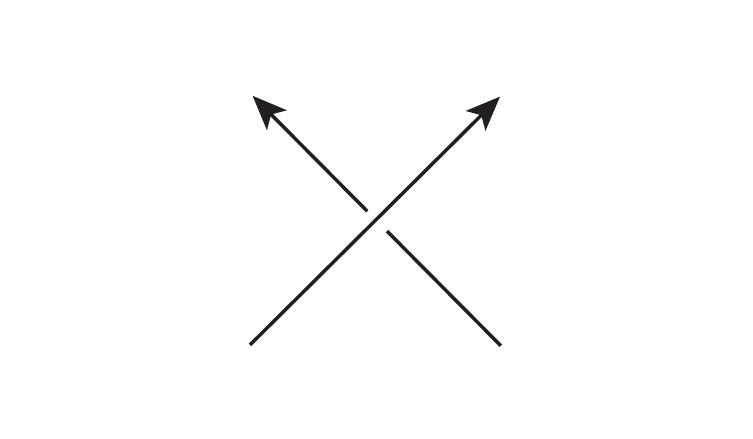}
\caption{\label{neg}}
 \end{minipage}
\end{figure}

In the proof of Theorem \ref{thm1},
we use the $d_1$-invariant of knots besides the $\tau$-invariant.
(The construction of the $d_1$-invariant is described in Section 2.)
We remark that if one wanted to prove only Corollary \ref{cor1}, 
one could use the knot signature $\sigma$ instead of the $d_1$-invariant. 
Furthermore, this alternative proof of Corollary \ref{cor1} enables us
to prove the following theorem.
\begin{thm}
\label{thm2}
For any knot $K$, there exist infinitely many knots
which cannot be deformed into $K$ either by only positive crossing
changes  or by only negative crossing changes.
\end{thm}
In the last section, we give such an alternative proof of Corollary \ref{thm1},
and a proof of Theorem \ref{thm2}.

Next we relate Corollary \ref{cor1} to the Gordian graph
$\mathcal{G}$ of knots (i.e., the 1-skeleton of the Gordian complex of knots \cite{hirasawa-uchida}).
The vertex set of $\mathcal{G}$ consists of (oriented) knot types in $S^3$, and 
two knot types are connected by an edge 
if they differ by a single crossing change.
We note that if we orient each edge from positive crossing to negative crossing,
we can naturally regard $\mathcal{G}$ as an oriented graph.
Then Corollary \ref{cor1} implies that there exist infinitely many knots
which are not contained in any flow that goes through some slice knot.

Finally, let us mention the kinkiness of knots as defined by Gompf \cite{gompf}.
Let $K$ be a knot in $S^3= \partial B^4$, and
consider all self-transverse immersed disks in $B^4$ with boundary $K$.
Then we define
$k_+(K)$ (resp.\ $k_-(K)$) to be the minimal number of positive (resp.\ negative) 
self-intersection points occurring in any such disk.
Gompf proved in \cite{gompf} that for any $n \in \mathbb{N}$,
there exists a topologically slice knot $K$ such that $(k_+(K),k_-(K))=(0,n)$.
On the other hand, as far as the author knows,
whether  there exist knots which satisfy $k_+>0$ and $k_->0$ 
has been unknown.
Theorem \ref{thm1} also implies that there exist infinitely many such knots.

\begin{cor}
\label{cor2}
For any knot $K$ with $\tau (K) <0$ and any integer $n$ with
$0 < n < -2\tau(K) + (\varepsilon(K) - 1)/2$, we have
 $k_+(K_{2,2n+1})>0$ and $k_-(K_{2,2n+1})>0$.
\end{cor}

\begin{remark}
Marco Golla has pointed out that 
in \cite[Example 5.1]{owens-strle} Owens and Strle show that the knot $T_{3,10}\#T_{5,-6}$ cannot bound a disk with only positive or negative self-intersections.
(Here $T_{p,q}$ denotes the $(p,q)$-torus knot.)
Hence the existence of a knot which satisfies  $k_+>0$ and $k_->0$ 
has been already known.
He also has pointed out that  
using \cite[Proposition 2.1]{owens-strle}, \cite[Corollary 5.1]{bodnar-celoria-golla} can be tweaked to work for the crossing number (giving lower bounds for the number of positive/negative self-intersections). Also, that same example, combined with Lemma 1 in this paper, provides many more examples of knots that are not slice either in positons or in  negatons.
\end{remark}

\begin{acknowledge}
The author was supported by JSPS KAKENHI Grant Number 15J10597.
The author would like to thank his supervisor, Tam\'{a}s K\'{a}lm\'{a}n
for his useful comments and encouragement.
The author also would like to thank Takuji Nakamura, Seiichi Kamada
and Marco Golla 
for their useful comments.
\end{acknowledge}

\section{Proof of Theorem \ref{thm1}}

In this section, we prove Theorem \ref{thm1}.
Before giving the proof,
we recall a few facts about the $d_1$-invariant.

Ozsv\'{a}th and Szab\'{o} \cite{ozsvath-szabo} introduced a $\mathbb{Q}$-valued invariant
$d$ (called the {\it correction term})
for rational homology 3-spheres associated with a Spin$^c$ structure.
In particular, since any integer homology 3-sphere $Y$ has a unique Spin$^c$ structure,
we may denote the correction term simply by $d(Y)$ in this case.
Furthermore, 
we note  that for any integer homology 3-sphere $Y$, $d_1(Y)$ is an even integer. 
Let $S^3_1(K)$ denote the $1$-surgery along  a knot $K$ in $S^3$.
Then $S^3_1(K)$ is an integer homology 3-sphere,
and hence we can define the {\it $d_1$-invariant} of $K$ as $d_1(K):=d(S^3_1(K))$.
It is known that $d_1(K)$ is a knot concordance invariant of $K$.
For details, see \cite{peters}. 

Next, we give two necessary conditions for the sliceness of knots in some negaton.
The first condition is with respect to the $d_1$-invariant.
\begin{lem}
\label{lem1}
If a knot $K$ is slice in some negaton,
then we have $d_1(K) = 0$.
\end{lem}

\begin{proof}
It is proved in \cite{peters} that $d_1(K) \leq 0$ for any knot $K$.
Hence we only need to show that $d_1(K) \geq 0$. 

Suppose that $K$ is slice in a negaton $V$. Then there exists a
properly embedded null-homologous disk $D$ in $V$ with boundary $K$.
By attaching a $(+1)$-framed 2-handle $h^2$ along $K$, and gluing $D$ with
the core of $h^2$, we obtain an embedded 2-sphere $S$ in $W := V \cup h^2$
with self-intersection $+1$. This implies that there exists a 4-manifold $W'$ 
with boundary $S^3_1(K)$ such that $W = W' \# \mathbb{C}P^2$. 
Note that $\partial W' = \partial W = S^3_1(K)$.
Since the number of positive eigenvalues of the intersection form of $W$ is one,
the intersection form of $W'$ must be negative definite.
Now we use the following theorem.
\begin{thm}[Ozsv\'{a}th-Szab\'{o}, \text{\cite[Corollary 9.8]{ozsvath-szabo}}]
\label{thm d_1}
If $Y$ is an integer homology 3-sphere with $d(Y)<0$, then there is
no negative definite 4-manifold $X$ with $\partial X = Y$.
\end{thm}
By Theorem \ref{thm d_1} and the existence of $W'$, we have $d_1(K) = d(S^3_1(K)) \geq 0$.
\end{proof}

We next show the other necessary condition, which is with respect to the $\tau$-invariant. 

\begin{lem}
\label{lem2}
If a knot $K$ is slice in some negaton,
then we have $\tau(K) \leq 0$.
\end{lem}

Indeed, Lemma \ref{lem2} immediately follows from the following theorem of
Ozsv\'{a}th-Szab\'{o}. 
Let $V$ be a negaton. According to Donaldson's celebrated theorem,
the intersection form of $V$ is diagonalizable.
Write a homology class $x \in H_2(V, \partial V;\mathbb{Z})$ as
$$
x = s_1 \cdot e_1 + \cdots + s_b \cdot e_b,
$$
where the $e_i$ form an orthonormal basis for $H_2(V,\partial V;\mathbb{Z})$,
and $s_i \in \mathbb{Z}$.
Then we define $\left| x \right| = |s_1| + \cdots + |s_b|$.

\begin{thm}[Ozsv\'{a}th-Szab\'{o}, \text{\cite[Theorem 1.1]{ozsvath-szabo2}}]
\label{thm tau}
Let $V$ be a negaton. If $\Sigma$ is a properly embedded surface
in $V$ with boundary a knot $K$, then we have 
$$
2\tau(K) + \left| [\Sigma, \partial \Sigma] \right| 
+ [\Sigma, \partial \Sigma] \cdot [\Sigma, \partial \Sigma] \leq 2g(\Sigma),
$$
where $g(\Sigma)$ denotes the genus of $\Sigma$.
\end{thm}

\def\proofname{Proof of Lemma \ref{lem2}}
\begin{proof}
Since $K$ is slice in some negaton $V$, there exists a properly embedded disk $D$ in $V$
with boundary $K$ such that $[D,\partial D] = 0 \in H_2(V,\partial V; \mathbb{Z})$. Obviously
we have $|[D,\partial D]|=0$, $[D,\partial D]\cdot[D,\partial D]=0$ and $g(D)=0$,
and by Theorem \ref{thm tau}, we obtain the inequality $2\tau(K) \leq 0$.
\end{proof}

In the case of the $\tau$-invariant, we can derive a necessary condition
for the sliceness of knots in some positon.

\begin{lem}
\label{lem3}
If a knot $K$ is slice in some positon,
then we have $\tau(K) \geq 0$.
\end{lem}

\def\proofname{Proof}
\begin{proof}
If there exists a slice disk in a positon $V$ with boundary $K$,
then by reversing the orientation of $V$,
we obtain a slice disk in $-V$ with boundary $K^*$,
where $K^*$ denotes the mirror image of $K$.
Since $-V$ is a negaton
and $\tau$ is a group homomorphism from the smooth knot concordance group
to $\mathbb{Z}$, Lemma \ref{lem2} gives 
$$
0 \geq \tau(K^*) = - \tau(K).
$$ 
\end{proof}

\def\proofname{Proof of Theorem \ref{thm1}}
\begin{proof}
Let $K$ be a knot with $\tau (K) <0$ and $n$ an integer with
$0 < n < -2\tau(K) + (\varepsilon(K) - 1)/2$.
In order to apply Lemma \ref{lem1} and Lemma \ref{lem3},
we estimate $\tau (K_{2,2n+1})$ and $d_1(K_{2,2n+1})$.
Indeed, such estimates have been already given by the author \cite{sato} and Hom \cite{hom}. 
The estimate of $d_1(K_{2,2n+1})$ is as follows.
\begin{thm}[\text{\cite[Theorem 1]{sato}}]
For any knot $K$ in $S^3$ and any $k \in \mathbb{N}$, we have
$$
d_1(K_{2,4k \pm 1}) \leq -2k.
$$
\end{thm}
In particular, we have $d_1(K_{2,2n+1}) <0$, and Lemma \ref{lem1}
implies that $K_{2,2n+1}$ cannot be slice in any negaton.
Next we derive an estimate of $\tau(K_{2,2n+1})$ from the following theorem.
\begin{thm}[Hom, \text{\cite[Theorem 1]{hom}}]
\label{thm.hom}
Let $K$ be a knot in $S^3$ and $p>1$. Then $\tau(K_{p,q})$
is determined in the following manner.
\begin{enumerate}
\item If $\varepsilon(K)=1$, then $\tau(K_{p,q}) = p\tau(K) + (p-1)(q-1)/2$.
\item If $\varepsilon(K)=-1$, then $\tau(K_{p,q}) = p\tau(K) + (p-1)(q+1)/2$.
\item If $\varepsilon(K)=0$, then $\tau(K_{p,q}) = \tau(T_{p,q})$.
\end{enumerate}
\end{thm}
In particular, Theorem \ref{thm.hom} implies that if $\varepsilon (K) \neq 0$,
we have the equality
$$
\tau(K_{p,q}) = p\tau(K) + (p-1)(q-\varepsilon(K))/2.
$$
In our case, note that $\varepsilon (K) \neq 0$ since $\tau(K) <0$,
and hence we have
$$
\tau(K_{2,2n+1}) = 2\tau(K) + (2-1)(2n+1-\varepsilon (K))/2 = 2\tau(K) + n - 
(\varepsilon (K) -1)/2 < 0.
$$
Thus by Lemma \ref{lem3},
$K_{2,2n+1}$ cannot be slice in any positon.
\end{proof}

\section{Proofs of corollaries}
In this section, we prove the two corollaries stated in Section 1.
We start with the following two propositions,
which imply that the two corollaries in Section 1 naturally follow from Theorem \ref{thm1}.
 
\begin{prop}
\label{prop1}
If a knot $K$ can be deformed into a slice knot by only positive (resp.\ negative) crossing changes, 
then $k_-(K) =0$ (resp.\ $k_+(K)= 0$).
\end{prop}

\def\proofname{Proof}
\begin{proof}
Suppose that a knot $K$ can be deformed into a slice knot 
by only positive crossing changes.
Then we can construct a self-transverse immersed disk
in $B^4$ with boundary $K$ which has no negative kinks, by assigning to each positive crossing changes the motion picture shown in Figure $\ref{crossing_kink}$ and capping off the 
resulting slice knot with its slice disk.
This implies that $k_-(K)=0$.
We can also prove in the same way that
if $K$ can be unknotted by only negative crossing changes, then $k_+(K)=0$.
\end{proof}

\begin{figure}[tbp]
\begin{center}
\includegraphics[scale = 0.7]{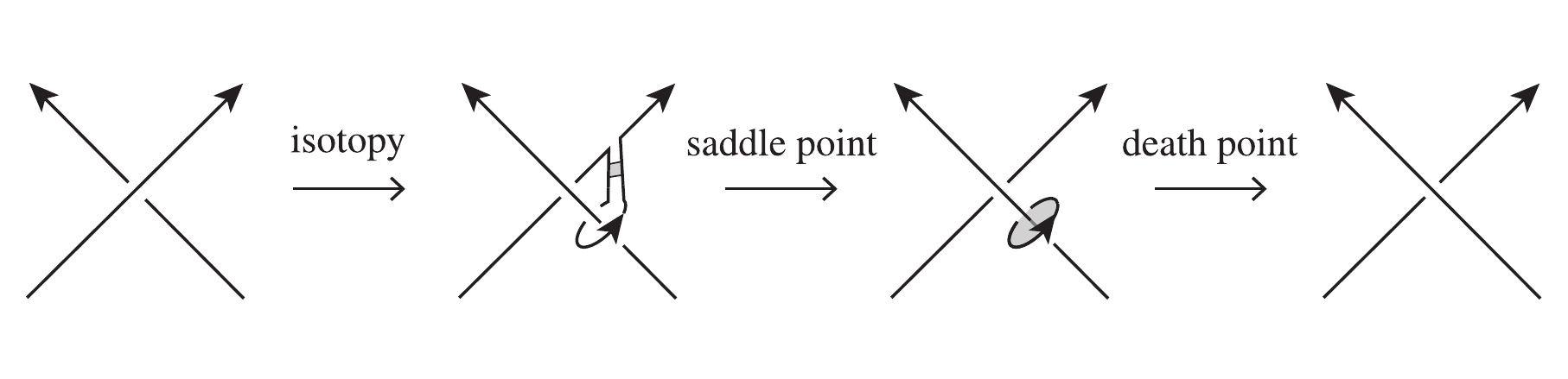}
\caption{\label{crossing_kink}}
\end{center}
\end{figure}

\begin{prop}
\label{prop2}
If $k_-(K) =0$ (resp.\ $k_+(K)= 0$), then $K$ is slice in 
$m\mathbb{C}P^2 \setminus \mathring{B^4}$
(resp.\ $m\overline{\mathbb{C}P^2} \setminus \mathring{B^4}$)
for some $m \in \mathbb{N}$.
\end{prop}

\begin{proof}
Let $D$ be a self-transverse immersed disk in a 4-ball $B$ with boundary $K$
which has no negative kink 
and $m$ positive kinks $p_1, p_2, \ldots, p_m$ $(m \in \mathbb{N} )$.
Then by taking a small open 4-ball $B_i$ containing $p_i$ $(i = 1, \ldots , m)$
and removing $\cup^{m}_{i=1}B_i$ from $B$, we obtain 
a properly embedded planar surface $S$ in $B \setminus \cup^{m}_{i=1}B_i$
with $2m+1$ boundary components
such that $S \cap \partial B = K$ and 
$(S \cap  \partial B_i) \subset \ \partial B_i$ is the positive Hopf link $H_+$ for any $i$.
Furthermore, the motion picture shown in Figure \ref{hopf_link} gives mutually distinct  
disks $E$ in $\mathbb{C}P^2 \setminus B^4$ with boundary $H_+$
such that $[E, \partial E] = 0 \in H_2(\mathbb{C}P^2 \setminus B^4 , S^3 ; \mathbb{Z})$.
Gluing $(\mathbb{C}P^2 \setminus B^4, E)$ with $(B \setminus \cup^{m}_{i=1}B_i, S)$
along $(\partial B_i, H_+)$ for each $i = 1, \ldots, m$,
we obtain a null-homologous disk $D'$ in $m\mathbb{C}P^2 \setminus B^4$
with boundary $K$. In particular, $K$ is slice in $m\mathbb{C}P^2 \setminus B^4$.
We can also prove in the same way that if $k_+(K) = 0$,
then $K$ is slice in $m \overline{\mathbb{C}P^2} \setminus B^4$ for some $m \in \mathbb{N}$.
\end{proof}

\begin{figure}[tbp]
\begin{center}
\includegraphics[scale = 0.7]{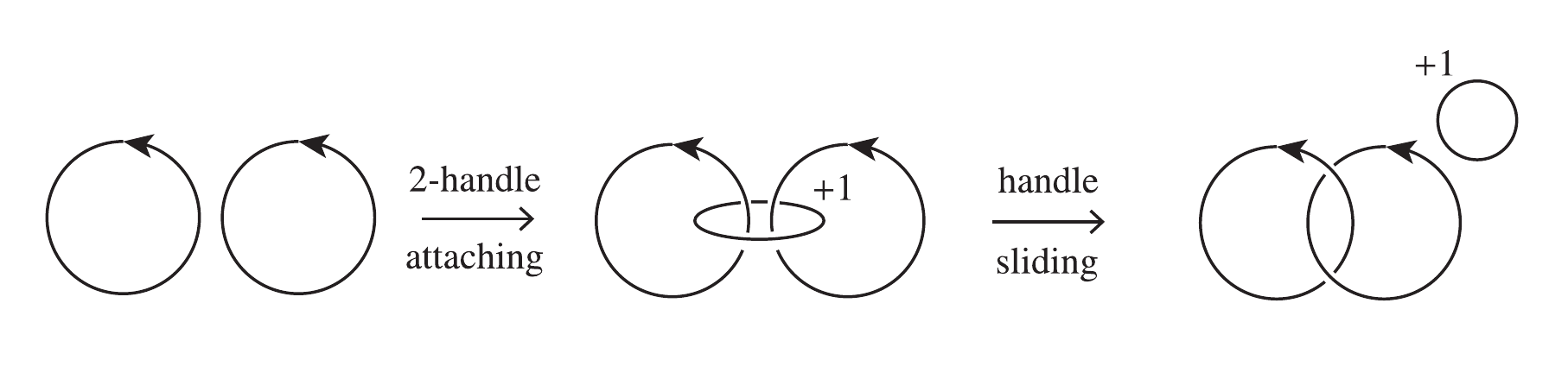}
\caption{\label{hopf_link}}
\end{center}
\end{figure}

\def\proofname{Proof of Corollary \ref{cor2}}
\begin{proof}
Let $K$ be a knot with $\tau (K) <0$ and $n$ an integer with
$0 < n < -2\tau(K) + (\varepsilon(K) - 1)/2$.
Then by Theorem \ref{thm1}, $K$ is not slice either in positons or in negatons.
Hence the contrapositive of Proposition \ref{prop2} implies
that  $k_+(K_{2,2n+1})>0$ and $k_-(K_{2,2n+1})>0$.
\end{proof}

\def\proofname{Proof of Corollary \ref{cor1}}
\begin{proof}
Let $K$ be a knot with $\tau (K) <0$ and $n$ an integer with
$0 < n < -2\tau(K) + (\varepsilon(K) - 1)/2$.
Then by Corollary \ref{cor2}, we have $k_+(K_{2,2n+1})>0$ and $k_-(K_{2,2n+1})>0$.
Hence the contrapositive of Proposition \ref{prop1} implies that
$K_{2,2n+1}$ cannot be deformed into a slice knot
either by only positive crossing changes or by only negative crossing changes.
\end{proof}

\section{Remark on Corollary \ref{cor1}}
In this section, we give an alternative proof of Corollary \ref{cor1} and
a proof of Theorem \ref{thm2}, by using the knot signature $\sigma$ instead of the
$d_1$-invariant.

Suppose that a knot $K_-$ is obtained from $K_+$ by a positive crossing change.
It is shown in \cite[Theorem 6.4.7]{murasugi} and \cite[Corollary 1.5]{ozsvath-szabo2}
that we have
$$
\sigma(K_-) -2 \leq \sigma(K_+) \leq \sigma(K_-)
$$
and
$$
\tau(K_+)-1 \leq \tau(K_-) \leq \tau(K_+).
$$
In particular, the following lemmas are derived from the above inequalities.

\begin{lem}
\label{lem4}
Let $K_1$ and $K_2$ be knots with $\sigma(K_1) < \sigma(K_2)$.
Then $K_1$ cannot be deformed into $K_2$ by only negative crossing changes.
\end{lem}

\begin{lem}
\label{lem5}
Let $K_1$ and $K_2$ be knots with $\tau(K_1) < \tau(K_2)$.
Then $K_1$ cannot be deformed into $K_2$ by only positive crossing changes.
\end{lem}

\def\proofname{Proof of Corollary \ref{cor1}}
\begin{proof}
By Lemma \ref{lem4}, Lemma \ref{lem5} and the fact that
$\sigma(U) = \tau(U)= 0$ for the unknot $U$,
we only need to prove that 
for any knot $K$ with $\tau (K) <0$ and any integer $n$ with
$0 < n < -2\tau(K) + (\varepsilon(K) - 1)/2$, 
the inequalities 
$\sigma(K_{2,2n+1})<0$ and $\tau(K_{2,2n+1})<0$ hold.
The inequality $\tau(K_{2,2n+1})<0$ has been already obtained in
the proof of Theorem \ref{thm1}.
Furthermore, 
the Litherland satellite formula for $\sigma$ in \cite[Theorem 2]{litherland}
implies that
$\sigma(K_{2,2n+1}) = \sigma(T_{2,2n+1}) = -2n$ for any knot $K$.
In particular, we have $\sigma(K_{2,2n+1}) <0$.
This completes the proof.
\end{proof}

\def\proofname{Proof of Theorem \ref{thm2}}
\begin{proof}
By Lemma \ref{lem4} and Lemma \ref{lem5},
we only need to prove that there exist infinitely many knots
$K_i$ $(i \in \mathbb{N})$ which satisfy $\sigma(K_i) < \sigma(K)$ 
and $\tau(K_i) < \tau(K)$.
Let $K_i := (T_{2,-2i-3})_{2,3} \# K$,
where $\#$ denotes the connected sum operation.
(Here $T_{2,-2i-3}$ is a concrete example of a knot which, with $n=1$, satisfies the conditions of Theorem \ref{thm1}.)
Since $\tau (T_{2,-2i-3}) = -i-1$ and $\varepsilon (T_{2,-2i-3}) = -1$, 
the inequality $0 < 1 < -2 \tau(T_{2,-2i-3}) + (\varepsilon(T_{2,-2i-3}) - 1)/2$
holds for any $i \in \mathbb{N}$.
Hence the argument in the above proof of Corollary \ref{cor1} implies that 
$\sigma((T_{2,-2i-3})_{2,3}) < 0$ and $\tau((T_{2,-2i-3})_{2,3}) < 0$.
Furthermore, since $\sigma$ and $\tau$ are additive,
we have
$$
\sigma(K_i)=\sigma((T_{2,-2i-3})_{2,3}) + \sigma(K) < \sigma(K)
$$
and
$$
\tau(K_i)=\tau((T_{2,-2i-3})_{2,3}) + \tau(K) < \tau(K).
$$
Note that 
$\Delta_{K_i}(t) = \Delta_{T_{2,3}}(t) \cdot \Delta_{T_{2,-2i-3}}(t^2) \cdot \Delta_{K}(t)$
and $\Delta_{T_{2,-2i-3}}(t^2) \neq \Delta_{T_{2,-2j-3}}(t^2)$ 
(for $i,j \in \mathbb{N}$, $i \neq j$),
and hence the $K_i$ are mutually distinct.
Here $\Delta$ denotes the Alexander polynomial.
\end{proof}

\begin{remark}
Note that $d_1$ also satisfies the skein inequality
$$
d_1(K_-) -2 \leq d_1(K_+) \leq d_1(K_-),
$$
which is proved by Peters \cite{peters}.
However, $d_1$ is not additive in general, 
and hence we cannot replace $\sigma$ with $d_1$ in the proof of Theorem \ref{thm2}.
\end{remark}

\end{document}